\documentclass[12pt,reqno]{amsart}
\usepackage[latin1]{inputenc}
\usepackage{amsmath}
\usepackage{amsfonts}
\usepackage{amssymb,esint}
\usepackage{mathrsfs}
\usepackage{hyperref}

\theoremstyle{plain}
\newtheorem{theorem}{Theorem}[section]
\newtheorem{lemma}[theorem]{Lemma}
\newtheorem{corollary}[theorem]{Corollary}

\theoremstyle{definition}
\newtheorem{definition}[theorem]{Definition}
\newtheorem{remark}[theorem]{Remark}
\newtheorem{example}[theorem]{Example}

\renewcommand{\theequation}{\arabic{section}.\arabic{equation}}

\newcommand{\R}{\mathbb R}
\newcommand{\N}{\mathbb N}
\renewcommand{\H}{\mathbb H}
\newcommand{\G}{\mathbb G}

\newcommand{\galg}{\mathfrak g}

\newcommand{\al}{\alpha}
\newcommand{\be}{\beta}

\newcommand{\de}{\delta}
\newcommand{\ep}{\varepsilon}

\newcommand{\la}{\lambda}
\newcommand{\si}{\sigma}
\newcommand{\Si}{\Sigma}
\newcommand{\tS}{\widetilde\Sigma}
\newcommand{\tp}{\widetilde\pi}

\usepackage{color}

\newcommand{\res}
{\mathop{\hbox{\vrule height 7pt width .5pt depth 0pt \vrule
height .5pt width 6pt depth 0pt}}\nolimits}

\newcommand{\tR}{\times\R}
\newcommand{\Shaus}{\mathcal S}
\newcommand{\tiR}{\widetilde R}
\newcommand{\W}{\mathbb W}
\newcommand{\prop}{\mathscr C}
\newcommand{\Rprop}{\mathscr R}
\newcommand{\f}{\phi}
\newcommand{\dive}{{\mathrm{div\!}_X\>\!}}

\long\def\MSC#1\EndMSC{\def\arg{#1}\ifx\arg\empty\relax\else
     {\par\narrower\noindent
     {\small\it 2010 Mathematics Subject Classification.} \small #1\par}\fi}

\long\def\KEY#1\EndKEY{\def\arg{#1}\ifx\arg\empty\relax\else
     {\par\narrower\noindent
     {\small\it Keywords and Phrases.} \small #1\par}\fi}

\author[Don]{Sebastiano Don}
\author[Massaccesi]{Annalisa Massaccesi}
\author[Vittone]{Davide Vittone}
\address[Massaccesi]{Dipartimento di Informatica, strada le Grazie 15, 37134 Verona, Italy.}
\email{annalisa.massaccesi@univr.it}
\address[Don and Vittone]{Dipartimento di Matematica, via Trieste 63, 35121 Padova, Italy.}
\email{sebastiano.don@math.unipd.it,vittone@math.unipd.it}
\thanks{S.D. and D.V. are supported by University of Padova, Project Networking, and GNAMPA of INdAM (Italy), project ``Campi vettoriali, superfici e perimetri in geometrie singolari''. A.M. is supported by University of Verona. This project has received funding from the European Union's Horizon 2020 research and innovation programme under grant agreement No 752018 (CuMiN).  
The authors also wish to thank the Institut f\"ur Mathematik of Zurich University for hospitality and support during the preparation of this paper.\\}



\begin{document}


\title
{Rank-one theorem and subgraphs of BV functions in Carnot groups}

\begin{abstract}
We prove a rank-one theorem {\em \`a la} G.~Alberti for the derivatives of vector-valued maps with bounded variation in a class of Carnot groups that includes Heisenberg groups $\H^n$ for $n\geq 2$. The main tools are  properties relating the horizontal derivatives of a real-valued function with  bounded variation and its subgraph.
\end{abstract}

\maketitle\vspace{-1mm}

\section{Introduction}
One of the main results in the theory of functions with bounded variation (BV) is the  rank-one theorem. Recall that a function $u\in L^1(\Omega,\R^d)$ has  bounded variation in an open set $\Omega\subset\R^n$ ($u\in BV(\Omega,\R^d)$) if the derivatives $Du$ of $u$ in the sense of distributions are represented by a (matrix-valued) measure with finite total variation. The measure $Du$ can then be decomposed as the sum $Du=D^au+D^su$ of a measure $D^au$, that is absolutely continuous with respect to $\mathscr L^n$, and a measure $D^su$ that is singular with respect to $\mathscr L^n$. The Radon-Nikodym derivative $\frac{D^su}{|D^su|}$ of $D^su$ with respect to its total variation $|D^su|$ is a $|D^su|$-measurable map from $\Omega$ to $\R^{d\times n}$. The rank-one theorem  states that $|D^su|$-a.e. this map takes values in the space of rank-one matrices. We refer to \cite{AFP} for more details on $BV$ functions.

The rank-one theorem was first conjectured by L.~Ambrosio and E.~De~Giorgi in \cite{ADGNuovoTipoFunz} and it has important  applications to vectorial variational problems and  systems of PDEs. It was  proved by G.~Alberti in \cite{Alb-RankOne} (see also \cite{AlbCso,DeL}): due to its complexity, Alberti's proof is generally regarded as a  {\em tour de force} in measure theory. Two different proofs of the rank-one theorem were recently found. One is due to G.~De~Philippis and F.~Rindler and  follows from a profound PDE result \cite{DPR}, where a  rank-one property for maps with bounded deformation (BD) was also proved for the first time.  At the same time another proof, of a geometric flavor and considerably simpler than those in \cite{Alb-RankOne,DPR}, was  provided by the second- and third-named authors in \cite{MV}.

Motivated by these results, in this paper we consider the following natural generalization. Let $X_1,\dots,X_m$ be linearly independent vector fields in $\R^n$, $m\leq n$, and let $u:\Omega\to\R^d$ be a function with {\em bounded $H$-variation} in an open set $\Omega\subset\R^n$, i.e., a vector valued function such that the distributional {\em horizontal} derivatives $D_Hu:=(X_1u,\dots,X_mu)$ are represented by a $d\times m$-matrix valued measure with finite total variation in $\Omega$; consider the singular part $D_H^su$ of $D_Hu$  with respect to $\mathscr L^n$. Is it true that the Radon-Nikodym  derivative $\frac{D_H^su}{|D_H^su|}$ is  a rank-one matrix $|D_H^su|$-a.e.?

We  investigate this question in the  setting of {\em Carnot  groups} $\G\equiv\R^n$ (see Section \ref{sec:preliminari}) endowed with a left-invariant basis $X_1,\dots,X_m$ of the first layer $\galg_1$ in the stratification of their Lie algebra. In particular, we find two assumptions on $\G$, that we call properties $\prop_2$ and $\Rprop$  (see Definitions \ref{def:proprietaPk} and \ref{def:Rprop}, respectively), that ensure the rank-one property for $BV_H$ functions in $\G$. We will discuss later the role played by these properties in our argument. Our first main result is  the following 

\begin{theorem}\label{teo:rank1}
Let $\G$ be a Carnot group  satisfying properties $\prop_2$ and $\Rprop$; let  $\Omega\subset\G$ be an open set and $u\in BV_{H,loc}(\Omega,\R^d)$ be a function with locally bounded $H$-variation. Then the singular part $D_H^su$ of $D_Hu$ is a rank-one measure, i.e., the matrix-valued function $\frac{D_H^su}{|D_H^su|}(x)$ has rank one for $|D_H^su|$-a.e. $x\in\Omega$.
\end{theorem}

It is worth pointing out that Theorem  \ref{teo:rank1} applies to  the $n$-th {\em Heisenberg group} $\H^n$ provided $n\geq 2$. Recall that Heisenberg groups,  defined in  Example 	\ref{ex:introHn} below, are  the most notable examples of Carnot groups.

\begin{corollary}\label{cor:rank1Hn}
Let $u$ be as in Theorem  \ref{teo:rank1} and assume that $\G$ is the  Heisenberg group $\H^n$,  $n\geq 2$; then $D_H^su$  is a rank-one measure. More generally, the same holds if $\G$ is a  Carnot group of step 2 satisfying property $\prop_2$.
\end{corollary}

Corollary \ref{cor:rank1Hn} is an immediate consequence of  Theorem  \ref{teo:rank1}, see Remarks \ref{rem:esempiPk} and \ref{rem:step2Rprop}.

Theorem  \ref{teo:rank1} does not directly follow from the outcomes of \cite{DPR}, see Remark \ref{rem:rangounoH2}. Its proof follows the geometric strategy devised in \cite{MV} and it is based on the relations between a (real-valued)  $BV_H$ function $u$ in $\G$ and the {\em $H$-perimeter} of its subgraph $E_u:=\{(x,t):t<u(x)\}\subset\G\tR$. Recall that  a set $E\subset\G\tR$ has finite $H$-perimeter if its characteristic function $\chi_E$ has bounded $H$-variation with respect the vector fields of a basis of the first layer in the Lie algebra stratification of the Carnot group $\G\tR$. Our second main result is the following 

\begin{theorem}\label{teo:subgraphsemplificato}
Suppose that $\Omega\subset\G$ is open and bounded  and let $u \in L^1(\Omega)$. Then $u$ belongs to $BV_H(\Omega)$ if and only if its subgraph $E_u$ has finite $H$-perimeter in $\Omega\tR$. 
\end{theorem}

Actually, the proof of Theorem \ref{teo:rank1} requires much finer properties than the one stated in Theorem \ref{teo:subgraphsemplificato}. Such properties are stated in Theorems \ref{subgraph} and \ref{teo:normalevettorepolare} in a much more general context than Carnot groups, i.e., for maps with bounded $H$-variation with respect to a generic fixed family of linearly independent vector fields $X_1,\dots,X_m$ on $\R^n$.  Theorem  \ref{subgraph}, from which Theorem \ref{teo:subgraphsemplificato} immediately follows, focuses on the relations between the horizontal (in $\R^n$) derivatives of $u$ and the horizontal (in $\R^n\tR$) derivatives of $\chi_{E_u}$. Theorem \ref{teo:normalevettorepolare} instead deals with the relations between the {\em horizontal normal} to $E_u$ and the {\em polar vector} $\si_u$ in the decomposition $D_Hu=\si_u|D_Hu|$, and it also deals with the relations between $D_H^au,D_H^su$ and the horizontal  derivatives of $\chi_{E_u}$. When $m=n$ and $X_i=\partial_{x_i}$ one recovers some results that belong to the folklore of Geometric Measure Theory and  are  scattered in the literature (see e.g. \cite{MM}, \cite[4.5.9]{federer} and \cite[Section 4.1.5]{GMS}); we tried here to collect them in a more systematic way. We were not able to find references for some of the results we stated.

Property $\Rprop$ (``rectifiability'') intervenes in ensuring that the horizontal derivatives of $\chi_{E_u}$ are a ``rectifiable'' measure, see Definition \ref{def:Rprop}. This is a non-trivial technical obstruction one has to face when following the strategy of \cite{MV}: the rectifiability of sets with finite $H$-perimeter in Carnot groups is indeed a major open problem, which has been solved only in step 2 Carnot groups (see \cite{FSSCmathann,FSSCstep2}) and in the class of Carnot groups {\em of type $\star$} (\cite{MarcoMarchi}). See also \cite{AKLD} for a partial result in general Carnot groups. 

Once the rectifiability of $E_u$ is ensured, the proof of Theorem \ref{teo:rank1}  follows  rather easily  from the technical Lemma \ref{lem:trascurabilita} below, which is the natural counterpart of the Lemma in \cite{MV}. The latter, however, was proved by utilizing   the area formula for maps between rectifiable subsets of $\R^n$, see e.g. \cite{AFP}. A similar tool is not available in the context of Carnot groups, a fact which forces us to follow a different path.  The proof of Lemma \ref{lem:trascurabilita} is indeed achieved by a covering argument that is based on the following result: we  state it and  postpone to Section \ref{sec:preliminari} the definitions of property $\prop_k$, the Hausdorff measure $\mathcal H^d$, the homogeneous dimension $Q$ of $\G$ and of hypersurfaces of class $C^1_H$ with  their horizontal normal.

\begin{theorem}\label{teo:intersezione}
Let $k\geq 1$ be an integer,   $\G$  a Carnot group satisfying property $\prop_k$ and let $\Sigma_1,\dots,\Sigma_k$ be hypersurfaces of class $C^1_H$ with horizontal normals $\nu_1,\dots,\nu_k$. Let also $x\in \Si:=\Si_1\cap\dots\cap\Si_k$ be such that $\nu_1(x),\dots,\nu_k(x)$ are linearly independent. Then, there exists an open neighborhood $U$ of $x$ such that
\[
0<\mathcal H^{Q-k}(\Si\cap U)<\infty.
\]
In particular, the measure $\mathcal H^{Q-k}$ is $\sigma$-finite on the set
\[
\Si^{\pitchfork}:=\{x\in \Si :\nu_1(x),\dots,\nu_k(x)\text{ are linearly independent}\}.
\]
\end{theorem}

Theorem \ref{teo:intersezione}, that we prove in  Appendix \ref{sec:appendice}, is an easy consequence of Theorems \ref{teo:FSJGAmisurafinita} and \ref{teo:interse'grafLip} proved, respectively, in \cite{FSJGA} and \cite{MagTowards}. Theorem \ref{teo:interse'grafLip}, in particular, states the much deeper property that  $\Si^\pitchfork$ is locally  an {\em intrinsic Lipschitz graph}. To this aim, one needs the intersection  $T_x\Si_1\cap\dots\cap T_x\Si_k$ of the {\em tangent subgroups} to $\Si_i$ at $x$  to admit a (necessarily commutative) complementary  homogeneous  subgroup that is horizontal, i.e., contained in $\exp(\galg_1)$. This algebraic property is guaranteed by   property $\prop_k$ (``complementability''), see Remark \ref{rem:propkvssottogruppi}. We will provide in  Appendix \ref{sec:appendice} a proof of Theorem \ref{teo:interse'grafLip} which does not  rely on the homotopy invariance of the topological degree and is then simpler and shorter than the one in \cite{MagTowards}.


For the validity of Theorem \ref{teo:intersezione}, property $\prop_k$ might seem a restrictive one. We however point out that Theorem \ref{teo:intersezione} is no longer valid already when $k=2$ and $\G$ is the first Heisenberg group $\H^1$, which does not satisfy $\prop_2$: indeed, in this setting the measure $\mathcal H^{Q-2}(\Si^{\pitchfork})$ might be either 0 or $+\infty$ (even locally) as shown by A.~Kozhevnikov \cite{KozArxiv}. See also the recent paper \cite{MagSteTre}.

The fact that Theorem \ref{teo:intersezione} does not apply to $\H^1$ (actually, to $\H^1\tR\tR$, see the proof of Lemma \ref{lem:trascurabilita}) prevents us from proving the rank-one Theorem  \ref{teo:rank1} for $\G=\H^1$. This does not follow from \cite{DPR} either (see Remark \ref{rem:rangounoH1}) and, thus, it remains a very interesting open problem.\medskip

{\em Acknowledgements.} We are grateful to  G.~De~Philippis, U.~Menne and F.~Serra~Cassano for several stimulating discussions.

\section{Preliminaries on Carnot groups}\label{sec:preliminari}
\subsection{Algebraic facts}\label{subsec:algebra}
A {\em Carnot} (or {\em stratified}) {\em group} is a connected, simply connected and nilpotent Lie group whose Lie algebra $\galg$ is {\em stratified}, i.e., it has a decomposition $\galg=\galg_1\oplus\dots\oplus\galg_s$ such that
\[
\forall j=1,\dots,s-1\quad \galg_{j+1}=[\galg_j,\galg_1],\qquad\galg_s\neq\{0\}\qquad\text{and}\qquad [\galg_s,\galg]=\{0\}.
\]
We refer to the integer $s$ as the {\em step} of $\G$ and to $m:=\,$dim $\galg_1$ as its {\em rank}; apart from the case in which $\G$ is a Heisenberg group (see Example \ref{ex:introHn}), $n$  denotes the topological dimension of $\G$. The group identity is denoted by $0$. 

The exponential map $\exp:\galg\to\G$ is a diffeomorphism and, given a basis $X_1,\dots,X_n$ of $\galg$, we  often identify $\G$ with $\R^n$ by means of exponential coordinates:
\[
\R^n\ni x=(x_1,\dots,x_n)\longleftrightarrow \exp\left( x_1X_1+\dots+x_nX_n\right)\in\G.
\]
A one-parameter family $\{\de_\la\}_{\la>0}$ of {\em dilations} $\de_\la:\galg\to\galg$  is defined by $\de_\la(X):=\la^j X$ for any $X\in\galg_j$; notice that $\de_{\la\mu}=\de_\la\circ\de_\mu$. By composition with $\exp$ one can then define a one-parameter family, for which we use the same symbol $\de$, of group isomorphisms $\de_\la:\G\to\G$.

\begin{example}\label{ex:introHn}
Apart from Euclidean spaces, which are the only commutative Carnot groups, the most basic examples of Carnot groups are  Heisenberg groups. Given an integer $n\geq 1$, the $n$-th  Heisenberg group $\mathbb H^n$ is the $2n+1$ dimensional Carnot group of step 2 whose Lie algebra is generated by $X_1,\dots, X_n,Y_1,\dots,Y_n, T$ and the only non-vanishing commutation relations among these generators are given by
\[
[X_j,Y_j]=T\qquad\text{for any }j=1,\dots,n.
\]
The stratification of the Lie algebra is given by $\galg_1\oplus\galg_2$, where $\galg_1:=\mathrm{span}\{X_j,Y_j:j=1,\dots n\}$ and $\galg_2:=\mathrm{span}\{T\}$. In exponential coordinates
\[
\R^n\tR^n\tR\ni(x,y,t)\longleftrightarrow\exp(x_1X_1+\dots+y_nY_n+tT)
\]
one has
\[
X_j=\partial_{x_j}-\frac{y_j}2 \partial_t,\quad Y_j=\partial_{y_j}+\frac{x_j}2 \partial_t,\quad T=\partial_t.
\]
\end{example}

In this paper, given a Carnot group $\G$ we will frequently deal with products like $\G\tR^N$.  Needless to say, this is the Carnot group  with algebra $\galg\tR^N$  
with product  defined by $[(X,t),(Y,s)]=([X,Y],0)$ for any $X,Y\in\galg, t,s\in\R^N$ and whose stratification is given by $(\galg_1\times\R^N)\oplus(\galg_2\times\{0\})\oplus\dots\oplus(\galg_s\times\{0\})$.

\begin{definition}\label{def:proprietaPk}
Let $\G$ be a Carnot group with rank $m$ and let $1\leq k\leq m$ be an integer. We say that $\G$ satisfies the {\em property $\prop_k$} if the first layer $\galg_1$ of its Lie algebra has the following property: for any linear subspace $\mathfrak w$ of $\galg_1$ of codimension $k$ there exists a commutative complementary subspace in $\galg_1$, i.e., a $k$-dimensional subspace $\mathfrak h$ of $\galg_1$ such that $[\mathfrak h,\mathfrak h]=0$ and $\galg_1= \mathfrak w\oplus  \mathfrak h$.
\end{definition}

\begin{remark}\label{rem:propkvssottogruppi}
According to the terminology of Section \ref{sec:ipersuperfici}, a Carnot group has the property $\prop_k$ if and only if, for any vertical plane $\W$ in $\G$, there exists a complementary  homogeneous subgroup $\H$ that is horizontal, i.e., such that $\H\subset\exp(\galg_1)$. Notice also that, in this case, $\H$ is necessarily commutative.
\end{remark}

\begin{remark}\label{rem:esempiPk}
The Heisenberg group $\H^n$ has the property $\prop_k$ if and only if $1\leq k\leq n$. 

All Carnot groups have the property $\prop_1$. {\em Free} Carnot groups (see e.g. \cite{hall}) have the property $\prop_k$ if and only if $k=1$.

A Carnot group of rank $m$ has the property $\prop_m$ if and only if $\G$ is Abelian (i.e., $\G\equiv\R^m$). 
\end{remark}

\begin{remark}\label{rem:sottopropk}
It is an easy exercise to show that, if $k\geq 2$ and $\G$ has the propery $\prop_k$, then $\G$ has also the property $\prop_h$ for any $1\leq h\leq k$.
\end{remark}

\begin{lemma}\label{lem:propkprodotti}
Let $N\geq 1$ be an integer and $\G$ be a Carnot group. Then $\G$ has the property $\prop_k$ if and only if $\G\tR^N$ has the property $\prop_k$. 
\end{lemma}
\begin{proof}
It is clearly enough to prove the statement for $N=1$.

Assume first that $\G$ has the property $\prop_k$ and let $\mathfrak w$ be a $k$-codimensional subspace of the first layer $\galg_1\times\R$ of the Lie algebra of $\G\times\R$. We have two cases according to the dimension of $\mathfrak w':=\mathfrak w\cap(\galg_1\times\{0\})$:
\begin{itemize}
\item if dim $\mathfrak w'=m-k$, using the $\prop_k$ property of $\G$ one can find a $k$-dimensional commutative subspace $\mathfrak h$ of $\galg_1$ such that $\galg_1\times\{0\}=\mathfrak w'\oplus(\mathfrak h\times\{0\})$. In particular, $\galg_1\times\R=\mathfrak w\oplus(\mathfrak h\times\{0\})$;
\item if dim $\mathfrak w'=m+1-k$, then $\mathfrak w=\mathfrak w'\subset \galg_1\times\{0\}$ and, by Remark \ref{rem:sottopropk}, one can find a $(k-1)$-dimensional commutative subspace $\mathfrak h$ of $\galg_1$ such that $\galg_1\times\{0\}=\mathfrak w\oplus(\mathfrak h\times\{0\})$. In particular, $\galg_1\times\R=\mathfrak w\oplus(\mathfrak h\tR)$.
\end{itemize} 
In both cases we have found a commutative complementary subspace of $\mathfrak w$.

Assume now that $\G\tR$ has the property $\prop_k$ and let  $\mathfrak w$ be a $k$-codimensional linear subspace of $\galg_1$. Then $\mathfrak w\tR$ is a $k$-codimensional linear subspace of $\galg_1\tR$, hence it admits a $k$-dimensional commutative complementary subspace $\mathfrak h$ in $\galg_1\tR$. Denoting by $\pi:\galg_1\tR\to\galg_1$ the canonical projection, it is readily noticed that $\pi(\mathfrak h)$ is a $k$-dimensional commutative  subspace of $\galg_1$ such that $\galg_1=\mathfrak w\oplus \mathfrak \pi(\mathfrak h)$. This concludes the proof.
\end{proof}

\subsection{Metric facts}
Let $\G$ be a Carnot group with stratified algebra $\galg=\galg_1\oplus\dots\oplus\galg_s$. We  endow $ \galg$ with a positive definite scalar product $\langle\cdot,\cdot\rangle$ such that $\galg_i\perp \galg_j$ whenever $i\neq j$. We also let $|\cdot| :=\langle\cdot,\cdot\rangle^{1/2}$.  We fix an orthonormal basis $X_1,\dots,X_n$ of $\galg$ adapted to the stratification, i.e.,  such that  $\galg_j=\mathrm{span}\{ X_{m_{j-1}+1},\dots,X_{m_j}\}$ for any $j=1,\dots,s$, where  $m_j:=\mathrm{dim} (\galg_1)+\dots+\mathrm{dim}(\galg_j)$ and $m_0:=0$ (in particular, $m_1=m$).

We will frequently use the {\em homogeneous} (pseudo-){\em norm} $\|\cdot\|$ on $\G$ defined in this way: if  $x=\exp(Y_1+\dots+Y_s)$ for $Y_j\in\galg_j$, then
\[
\| x\| := \sum_{j=1}^s |Y_j|^{1/j}.
\]
Clearly one has $\|\de_\la(x)\|=\la\|x\|$ for any $x\in\G,\la>0$. Homogeneous pseudo-norms arising from different choices of the scalar product $\langle\cdot,\cdot\rangle$ on $\G$ are  equivalent.

The group $\G$ is endowed with the {\em Carnot-Carath\'eodory (CC) distance} $d$ induced by the family $X_1,\dots, X_m$, as we now introduce. Given an interval   $I\subset\R$, a Lipschitz curve $\gamma:I\to\G$ is said to be \emph{horizontal} if there exist  functions $h_1,\ldots, h_m \in L^\infty(I)$ such that for a.e. $t\in I$ we have
\begin{equation}\label{horiz}
	\dot \gamma (t) = \sum_{i=1}^m h_i(t) X_i (\gamma(t)).
\end{equation}
Letting $|h|:=(h_1^2+\ldots+h_m^2)^{1/2}$, the length of $\gamma$ is defined as
\[
	L(\gamma):=\int_I |h(t)|\,dt.
\]
It is well-known that for  any pair of points $x,y \in \G$ there exists a horizontal curve
joining $x$ to $y$. We can therefore define a distance function $d$ letting
\[
	d(x,y) := \inf \big\{ L(\gamma) : \gamma: [0,T]\to M \textrm{ horizontal with
	$\gamma(0) = x$ and $\gamma(T)=y$}\big\}.
\]
It is also well-known that, for any pair  $x,y \in \G$, there exists a geodesic joining $x$ and $y$, i.e., a horizontal curve $\gamma$ realizing the infimum in the previous formula. Notice that
\[
d(zx,zy)=d(x,y)\quad\text{and}\quad d(\delta_\lambda(x),\delta_\lambda(y))=\lambda d(x,y)\qquad\forall\:x,y,z\in G,\lambda>0
\]
and that $d(x,y)$ is equivalent to $\|x^{-1}y\|$.

We denote by $B(x,r)$ open balls of center $x\in\G$ and radius $r>0$ with respect to the CC distance; we also write $B_r$ instead of $B(0,r)$, so that $B(x,r)=x B_r$. The diameter $\textrm{diam }E$ of $E\subset\G$ and the distance $d(E_1,E_2)$ between $E_1,E_2\subset\G$ is understood with respect to the CC distance. 

As customary, for $E\subset\G,d>0$ and $\de>0$ we set
\[
\begin{split}
&\mathcal H^d_\de(E):=\inf\left\{\sum_{i=1}^\infty(\textrm{diam }E_i)^d:E\subset\bigcup_{i=1}^\infty E_i, \textrm{ diam }E_i<\de\right\}\\
&\Shaus^d_\de(E):=\inf\left\{\sum_{i=1}^\infty(\textrm{diam }B_i)^d:B_i\text{ are open balls},E\subset\bigcup_{i=1}^\infty B_i, \textrm{ diam }B_i<\de\right\}
\end{split}
\]
and we define the {\em $d$-dimensional Hausdorff measure} and {\em $d$-dimensional spherical Hausdorff measure} of $E$ respectively as
\[
\begin{split}
&\mathcal H^d(E):=\lim_{\de\downarrow 0} \mathcal H^d_\de(E)=\sup_{\de>0} \mathcal H^d_\de(E)\\
&\Shaus^d(E):=\lim_{\de\downarrow 0} \Shaus^d_\de(E)=\sup_{\de>0} \Shaus^d_\de(E).
\end{split}
\]
The Hausdorff dimension of  $E$ is $\inf\{d:\mathcal H^d(E)=0\}=\sup\{d:\mathcal H^d(E)=\infty\}$.  It is well-known that the metric space $(\G,d)$ has Hausdorff dimension $Q:=\sum_{j=1}^s j\:\mathrm{dim\ }\galg_j$ and that, in exponential coordinates and up to multiplicative constants, the measures $\mathcal H^Q$, $\Shaus^Q$ and $\mathscr L^n$ coincide, all of them being Haar measures on $\G$.

\section{Intrinsic regular hypersurfaces in Carnot groups}\label{sec:ipersuperfici}
We say that a continuous real function $f$ on an open set $\Omega\subset\G$ is {\em of class $C^1_H$} if its {\em horizontal derivatives} $X_1f,\dots,X_mf$ are continuous in $\Omega$. In this case we write $f\in C^1_H(\Omega)$ and we set $\nabla_H f:=(X_1f,\dots,X_mf)$.  

A set $S\subset\G$ is a {\em $C^1_H$ hypersurface} if for any $x\in S$ there exist an open neighborhood $U$ of $x$ and $f\in C^1_H(U)$ such that 
\[
S\cap U=\{y\in U:f(y)=0\}\quad\text{and}\quad \nabla_Hf\neq0\text{ on }U.
\]
In this case, we define the {\em horizontal normal} to $x$ as $\nu_S(x):=\frac{\nabla_Hf(x)}{|\nabla_Hf(x)|}\in\R^m$. The normal $\nu_S(x)=((\nu_S(x))_1,\dots,(\nu_S(x))_m)$ is defined up to sign and it can be canonically identified with a horizontal vector at $x$ by
\[
\nu_S(x)= (\nu_S(x))_1X_1(x)+\dots+(\nu_S(x))_mX_m(x).
\]
A $C^1_H$ hypersurface has locally finite $\mathcal H^{Q-1}$-measure, see e.g. \cite{VAnnSNS} and the references therein.\footnote{Actually, this also follows from Theorem \ref{teo:intersezione} with $k=1$.}

The hyperplane $\nu_S(x)^\perp$ in $\galg$ is a Lie subalgebra. The associated subgroup $T_xS:=\exp(\nu_S(x)^\perp)$ is called {\em tangent subgroup} to $S$ at $x$: we point out  the well-known property that
\begin{equation}\label{eq:blowup}
\forall\,\ep>0\ \exists\,\bar r=\bar r(x,\ep)>0 \text{ such that }\forall\, r\in(0,\bar r)\quad(x^{-1}S)\cap B_r \subset(T_xS)_{\ep r}\cap B_r,
\end{equation}
where for $E\subset\G$ and $\de>0$ we denote by $E_\de$ the $\de$-neighborhood of $E$.
A proof of \eqref{eq:blowup}, using the fact that in exponential coordinates $T_xS=\{(\xi,\eta)\in\R^n=\R^m\tR^{n-m}:\xi\perp\nu_S(x)\}$, is implicitly contained in the proof of Lemma \ref{lem:estensioneC1H}. Notice also that
\[
T_xS=\exp(\{X\in\galg_1:Xf(x)=0 \}\oplus\galg_2\dots\oplus\galg_s );
\]
in particular, while $\nu_S(x)$ depends on the scalar product $\langle\cdot,\cdot\rangle$  on $\galg$, the subgroup $T_xS$ is intrinsic. 

The tangent group $T_xS$ is a vertical plane of codimension 1 (or {\em vertical hyperplane}), where we say that $\W\subset\G$ is a {\em vertical plane} of codimension $k,1\leq k\leq m$, if $\W=\exp(\mathfrak w\oplus\galg_2\oplus\dots\oplus\galg_s)$ for some linear subspace $\mathfrak w$ of $\galg_1$ of codimension $k$ (possibly $\mathfrak w=\{0\}$). Such a $\W$ is a homogeneous normal subgroup of $\G$ of topological dimension $n-k$ and Hausdorff dimension $Q-k$. The intersection of vertical planes is always a vertical plane.

The following simple lemma will be used in the proof of Lemma \ref{lem:trascurabilita}.

\begin{lemma}\label{lem:ricoprimentopiani}
Let $\W\subset\G$ be a vertical plane of codimension $k$ and let $x\in\W$, $r>0$ and $\ep\in(0,1)$ be fixed. Then, the set $\W\cap B(x,r)$ can be covered by a family of balls $\{B(y_\ell,\ep r)\}_{\ell \in L}$ of radius $\ep r$ with cardinality $\# L\leq (4/\ep)^{Q-k}$.
\end{lemma}
\begin{proof}
By dilation and translation invariance, it is not restrictive to assume that $x=0$ and $r=1$. Let $\{y_\ell\}_{\ell\in L}$ be a maximal family of points of $\W\cap B(0,1)$ such that the balls $B(y_\ell,\ep/2)$ are pairwise disjoint; working by contradiction, it can be easily seen that the family $\{B(y_\ell,\ep)\}_{\ell \in L}$ covers $\W\cap B(0,1)$. The measure $\mathcal H^{Q-k}$ is locally finite on $\W$ (see e.g. \cite{M_nonhor,MagV,MTV}), is left-invariant and it is $(Q-k)$-homogeneous with respect to dilations. In particular, setting $M:=\mathcal H^{Q-k}(\W\cap B(0,1))$, we have
\[
\left(\frac{\ep}{2}\right)^{Q-k} M\:\#L =  \sum_{\ell\in L} \mathcal H^{Q-k}(\W\cap B(y_\ell,\ep/2 ))
\leq \mathcal H^{Q-k}(\W\cap B(0,2)) = 2^{Q-k}M,
\]
which proves the claim.
\end{proof}

A key tool in the proof of the rank-one Theorem \ref{teo:rank1} is the following Lemma \ref{lem:trascurabilita} which, in turn, uses Theorem \ref{teo:intersezione}, whose proof is instead postponed to Appendix \ref{sec:appendice}. We denote by $\pi:\G\tR\to\G$  the canonical projection $\pi(x,t)=x$.

\begin{lemma}\label{lem:trascurabilita}
Let $\G$ be a Carnot group satisfying property $\prop_2$. Let $\Sigma_1,\Sigma_2$ be $C^1_H$  hypersurfaces in $\G\tR$ with unit normals $\nu_{\Sigma_1},\nu_{\Sigma_2}$. Then, the set
\[
R:=\left\{p\in \Sigma_1: 
\exists\,q\in\Sigma_2\text{ such that }
\begin{array}{l}
\pi(q)=\pi(p),\\
\left(\nu_{\Sigma_1}(p)\right)_{m+1}=\left(\nu_{\Sigma_2}(q)\right)_{m+1}=0,\\
\nu_{\Sigma_1}(p)\neq\pm\nu_{\Sigma_2}(q)
\end{array}
\right\}
\]
is $\mathcal H^Q$-negligible.
\end{lemma}
\begin{proof}
Let us consider the distances $d_{\G\tR}$ and $d_{\G\tR\tR}$ on (respectively) $\G\tR$ and $\G\tR\tR$  defined by
\begin{align*}
& d_{\G\tR}((x,t),(x',t')):=d(x,x')+|t-t'|&&\forall\:x,x'\in\G,t,t'\in\R\\
& d_{\G\tR\tR}((x,t,s),(x',t',s')):=d(x,x')+|t-t'|+|s-s'|&&\forall\:x,x'\in\G,t,t',s,s'\in\R,
\end{align*}
where $d$ is the Carnot-Carath\'eodory distance on $\G$. Such distances are left-invariant and homogeneous, hence they are equivalent to the Carnot-Carath\'eodory distances on $\G\tR$ and $\G\tR\tR$; in particular, it is enough to prove the statement when the Hausdorff measure $\mathcal H^Q$ is the one induced by $d_{\G\tR}$ on $\G\tR$. We  use the same notation $B(a,r)$  for balls of radius $r>0$ in either $\G,\G\tR$ or $\G\tR\tR$, according to which group the center $a$ belongs to.

The sets 
\begin{align*}
& \tS_1 := \{(x,t,s)\in\G\tR\tR:(x,t)\in\Si_1,s\in\R\}\\
& \tS_2 := \{(x,t,s)\in\G\tR\tR:(x,s)\in\Si_2,t\in\R\}
\end{align*}
are clearly $C^1_H$ hypersurfaces in $\G\tR\tR$ and, moreover,
\begin{align*}
& \nu_{\tS_1}(x,t,s)=\big((\nu_{\Si_1}(x,t))_1,\dots, (\nu_{\Si_1}(x,t))_{m},(\nu_{\Si_1}(x,t))_{m+1},\ 0\ \big)\\
& \nu_{\tS_2}(x,t,s)=\big((\nu_{\Si_2}(x,s))_1,\dots,(\nu_{\Si_2}(x,s))_{m},\ 0\ , (\nu_{\Si_2}(x,s))_{m+1}\big).
\end{align*}
Let us define
\begin{align*}
\tiR:= & 
\{P\in\tS_1\cap\tS_2:(\nu_{\tS_1}(P))_{m+1}=(\nu_{\tS_2}(P))_{m+2}=0\text{ and }\nu_{\tS_1}(P)\neq\pm \nu_{\tS_2}(P)\}\\
=& \{(x,t,s)\in\tS_1\cap\tS_2:(\nu_{\Si_1}(x,t))_{m+1}=(\nu_{\Si_2}(x,s))_{m+1}=0\text{ and }\nu_{\Si_1}(x,t)\neq\pm \nu_{\Si_2}(x,s)\}.
\end{align*}
By construction we have $\tp(\tiR)=R$, where $\tp:\G\tR\tR\to\G\tR$ is the group homomorphism defined by $\tp(x,t,s):=(x,t)$; moreover the measure $\mathcal H^Q\res\tiR$ is  $\sigma$-finite by Theorem \ref{teo:intersezione} (notice that we are also using Lemma \ref{lem:propkprodotti}). We are going to show that $ \mathcal H^Q(\tp(T))=0$ for any fixed $T\subset\tiR$ such that $\Shaus^Q(T)<\infty$; this is clearly enough to conclude.

For any $P\in T$ and $i=1,2$, the tangent space $T_P\tS_i$ equals $\W_i\tR\tR$ for a suitable vertical hyperplane $\W_i$ of $\G$. In particular, setting $\W=\W(P):=\W_1\cap\W_2$, we have by \eqref{eq:blowup} that for any $P\in T$ and any $\ep\in(0,1)$ there exists $\bar r=\bar r(\ep, P)>0$ such that
\begin{equation}\label{eq:(2)}
\begin{split}
(P^{-1}T) \cap B(0,r)\subset & (\W\tR\tR)_{\ep r}\cap B(0,r) \\
= & (\W_{\ep r}\tR\tR)\cap B(0,r)\qquad \text{for any }r\in(0,\bar r).
\end{split}
\end{equation}
Notice also that $\W$ is a vertical plane of codimension 2 in $\G$. Let $\ep>0$ be fixed and set
\[
T_j:=\{P\in T: \bar r(\ep,P)\geq \tfrac 1j\},\qquad j=1,2,\dots
\]
Since $T_j\uparrow T$, the proof will be accomplished by showing that for any fixed $j$
\begin{equation}\label{eq:(3)}
\mathcal H^Q(\tp(T_j)) < C\ep,
\end{equation}
where $C>0$ is a constant that will be determined in the sequel.

Let us prove \eqref{eq:(3)}. Fix $\de\in(0,\tfrac 1j)$; since $\mathcal H^Q(T_j)\leq \mathcal H^Q(T)<+\infty$, one can find a (countable or finite) family $\{B(\widetilde P_i,r_i/2)\}_i$ of balls in $\G\tR\tR$ such that $0<r_i<\de$,
\[
T_j\subset \bigcup_i B(\widetilde P_i,r_i/2)\qquad\text{and}\qquad \sum_i (r_i/2)^Q \leq \sum_i (\text{diam } B(\widetilde P_i,r_i/2))^Q \leq C_1
\]
where $C_1:=\mathcal H^Q(T)+1$. We can also assume that $T_j\cap B(\widetilde P_i,r_i/2)$ is non-empty for any $i$. Choosing $P_i\in T_j\cap B(\widetilde P_i,r_i/2)$, for any $i$ the balls $B(P_i,r_i)$ have then the following properties:
\begin{equation}\label{eq:(333)}
\begin{split}
P_i\in T_j,\quad 0<r_i<\de,\quad T_j\subset \bigcup_i B(P_i,r_i)\quad\text{and}\quad \sum_i r_i^Q \leq 2^Q C_1.
\end{split}
\end{equation}
Setting $\W_i:=\W(P_i)$, by \eqref{eq:(2)} we have
\begin{equation}\label{eq:(4)}
\begin{split}
(P_i^{-1}T_j) \cap B(0,r_i)\subset & ((\W_i)_{\ep r_i}\tR\tR)\cap B(0,r_i)\\
= & ((\W_i)_{\ep r_i}\cap B(0,r_i))\times(-r_i,r_i)\times(-r_i,r_i).
\end{split}
\end{equation}
By Lemma \ref{lem:ricoprimentopiani}, for any $i$ we can find a family of balls $\{B(y_{i,\ell},\ep r_i)\}_{\ell\in L_i}$ such that 
\[
\forall\,\ell\in L_i\ y_{i,\ell}\in \W_i,\quad \# L_i\leq (8/\ep)^{Q-2}\quad\text{and}\quad \W_i\cap B(0,2r_i)\subset \bigcup_{\ell\in L_i} B(y_{i,\ell},\ep r_i).
\]
In particular
\begin{equation}\label{eq:(44)}
(\W_i)_{\ep r_i}\cap B(0,r_i) \subset (\W_i\cap B(0,r_i+\ep r_i))_{\ep r_i}\subset\bigcup_{\ell\in L_i} B(y_{i,\ell},2\ep r_i).
\end{equation}
Let us also fix points $\{\tau_k\}_{k\in K_i}\subset (-r_i,r_i)$ such that $\# K_i\leq 2\ep^{-1}$ and
\begin{equation}\label{eq:(444)}
(-r_i,r_i) \subset \bigcup_{k\in K_i} (\tau_k-2\ep r_i,\tau_k+2\ep r_i)
\end{equation}
By \eqref{eq:(4)}, \eqref{eq:(44)} and \eqref{eq:(444)} we get
\[
(P_i^{-1}T_j) \cap B(0,r_i)\subset \bigcup_{\substack{\ell\in L_i\\ k,h\in K_i}} B(y_{i,\ell},2\ep r_i)\times (\tau_k-2\ep r_i,\tau_k+2\ep r_i)\times (\tau_h-2\ep r_i,\tau_h+2\ep r_i).
\]
For any $\ell\in L_i$ and $k,h,h'\in K_i$ one has
\[
\begin{split}
& \tp\big( B(y_{i,\ell},2\ep r_i)\times (\tau_k-2\ep r_i,\tau_k+2\ep r_i)\times (\tau_h-2\ep r_i,\tau_h+2\ep r_i)\big)\\
=\:& \tp\big( B(y_{i,\ell},2\ep r_i)\times (\tau_k-2\ep r_i,\tau_k+2\ep r_i)\times (\tau_{h'}-2\ep r_i,\tau_{h'}+2\ep r_i)\big)\\
=\:& B(y_{i,\ell},2\ep r_i)\times (\tau_k-2\ep r_i,\tau_k+2\ep r_i)\\
=\:& B((y_{i,\ell},\tau_k),2\ep r_i)
\end{split}
\]
which, using \eqref{eq:(333)}, implies that
\[
\begin{split}
\tp(T_j) \subset & \bigcup_i \tp\big( T_j\cap B(P_i,r_i)\big)\\
\subset & \bigcup_i \bigcup_{\substack{\ell\in L_i\\ k,h\in K_i}} \tp\big( P_i(B(y_{i,\ell},2\ep r_i)\times (\tau_k-2\ep r_i,\tau_k+2\ep r_i)\times (\tau_h-2\ep r_i,\tau_h+2\ep r_i))\big)\\
= & \bigcup_i \bigcup_{\substack{\ell\in L_i\\ k\in K_i}} \tp (P_i)B((y_{i,\ell},\tau_k),2\ep r_i) \\
= & \bigcup_i \bigcup_{\substack{\ell\in L_i\\ k\in K_i}} B(p_{i\ell k},2\ep r_i)
\end{split}
\]
where $p_{i\ell k}:=\tp (P_i)(y_{i,\ell},\tau_k)\in\G\tR$. Using again \eqref{eq:(333)} we obtain that
\[
\mathcal H^Q_{2\ep \de}(T_j)\leq \sum_i\#L_i\ \#K_i\ (4\ep r_i)^Q
\leq 
\sum_i 2^{5Q-5}\ep  r_i^Q \leq  
2^{6Q-5}C_1 \ep
\]
which, by the arbitrariness of $\de\in(0,\tfrac1j)$, gives the claim  \eqref{eq:(3)}.
\end{proof}

\section{Functions with bounded \texorpdfstring{$H$}{H}-variation and subgraphs}\label{sec:sottografici}
Let $X=(X_1,\dots,X_m)$ be an $m$-tuple of linearly independent vector fields in $\mathbb{R}^n$;  for  $i=1,\dots,m$ and  $j=1,\dots,n$ we consider smooth functions $a_{ij} $  such that
\[
X_i(x)=\sum_{j=1}^n a_{ij}(x)\partial_{x_j}.
\] 
The model case is of course that of a Carnot group $\G\equiv\R^n$ endowed with a left-invariant basis $X_1,\dots,X_m$  of the first layer $\galg_1$ in the Lie algebra stratification; in the present section, however, we work in  higher generality.

One of the main purposes of this paper is the study of {\em functions with bounded $H$-variation} (\cite{capdangar,FSSChouston}), that we are  going to introduce only very briefly. In this section, $\Omega$ is an open subset of $\R^n$ and, given $\varphi\in C^1(\Omega, \mathbb{R}^m)$, we let $\dive\varphi:=\sum_{i=1}^mX_i^*\varphi_i$ where $X_i^*$ denotes the formal adjoint operator of the vector field $X_i$. Given a $\R^m$-valued function $f$ on $\Omega$ and a $\R^m$-valued measure $\mu$ on $\Omega$ we use the compact notation $\int_\Omega f\cdot d\mu$ for the sum $\int_\Omega f_1 \,d\mu_1+\dots+\int_\Omega f_m \,d\mu_m$.
	
\begin{definition}
We say that $u\in L^1_{loc}(\Omega)$ is a function of {\em locally bounded $H$-variation} in $\Omega$, and we write $u \in BV_{H,loc}(\Omega)$, if there exists a vector valued Radon measure $D_Hu=(D_{X_1}u,\dots,D_{X_m}u)$ with locally finite total variation such that for every $\varphi\in C_c^1(\Omega;\mathbb{R}^m)$ we have
\begin{equation}
\int_\Omega \varphi\cdot d\,D_{H}u=-\int_\Omega u\; \dive \varphi\, d\mathscr{L}^n.
\end{equation}

\noindent Moreover, if $u\in L^1(\Omega)$,  we say that $u$ has {\em bounded $H$-variation} in $\Omega$ ($u\in BV_H(\Omega)$) if  $D_Hu$ has finite total variation  $|D_Hu|$ on $\Omega$.

\noindent We say that $E\subset\Omega$ has {\em finite $H$-perimeter} in $\Omega$ if its characteristic function $\chi_E$ belongs to $BV_H(\Omega)$. 
\end{definition}

We recall that the total variation $|\mu|$ of a $\R^d$-valued measure $\mu=(\mu_1,\dots,\mu_d)$ is defined for Borel sets $B$ as
\[
\begin{split}
|\mu|(B):= & \sup\left\{ \sum_{\ell=1}^\infty |\mu(B_\ell)|:(B_\ell)_\ell\text{ disjoint Borel subsets of }B \right\}\\
=& \sup\left\{ \int_B\varphi\cdot d\mu:\ \varphi\!:\!B\to\R^d\text{ Borel function, }|\varphi|\leq1\right\}.
\end{split}
\]
If $A\Subset \Omega$ is open and $u \in BV_{H,loc}(\Omega)$, one can easily prove that 
\[
|D_Hu|(A) = \sup \left\{ \int_A u\; \dive \varphi \,d\mathscr{L}^n: \varphi\in C_c^1(A;\mathbb{R}^m),|\varphi|\leq 1\right\};
\]
actually, $u\in BV_H(A)$ if and only if the supremum on the right-hand side is finite. The total variation is lower-semicontinuous with respect to the $L^1_{loc}$ convergence; moreover (see \cite{GN,FSSChouston}), for any $u\in BV_H(\Omega)$ there exists a sequence $(u_h)_h$ in $ C^\infty(\Omega)\cap BV_H(\Omega)$ such that
\begin{equation}\label{eq:smoothapproxstrict}
\begin{split}
& u_h\to u\text{ in }L^1(\Omega)\\
& |D_Hu_h|(\Omega)\to |D_Hu|(\Omega)\\
&  |D_{X_i}u_h|(\Omega)\to |D_{X_i}u|(\Omega)\quad \forall\,i=1,\dots,m\\
& |(D_Hu_h,\mathscr{L}^n)|(\Omega)\to |(D_Hu,\mathscr{L}^n)|(\Omega).
\end{split}
\end{equation}

The aim of this section is the study of the relations occurring between a function $u\in BV_H(\Omega)$ and its {\em subgraph}
\[
E_u:=\{(x,t)\in\Omega\tR:t<u(x)\}\subset\Omega\tR.
\]
We  introduce the family $\widetilde X=(\widetilde X_1,\dots,\widetilde X_{m+1})$ of linearly independent vector fields in $\R^{n+1}$ defined for $(x,t)\in\R^n\tR$ by
\[
\begin{split}
& \widetilde{X}_i(x,t):=(X_i(x),0)\in\R^{n+1}\equiv\R^n\tR\qquad\text{if }i=1,\dots, m\\
& \widetilde X_{m+1}(x,t):=\partial_t.
\end{split}
\]
If $U\subset\R^{n+1}$ is open and $u\in BV_{H,loc}(U)$ with respect to the family $\widetilde X$ we write $D_{\widetilde{H}}u:=(D_{\widetilde{X}_1}u,\dots,D_{\widetilde{X}_{m+1}}u)$.

The following result is the natural generalization of some classical facts about Euclidean functions of bounded variation, see e.g. \cite[Section 4.1.5]{GMS}. We denote by $\pi:\mathbb{R}^{n+1}\rightarrow\mathbb{R}^n$ the canonical projection $\pi(x,t)=x$; $\pi_\#$ denotes the associated push-forward of measures.

\begin{theorem}\label{subgraph}
Suppose $\Omega$ is bounded in $\mathbb{R}^n$ and let $u \in L^1(\Omega)$. Then $u$ belongs to $BV_H(\Omega)$ if and only if its subgraph $E_u$ has finite $H$-perimeter (with respect to the family $\widetilde X$) in $\Omega\tR$. 
		
\noindent Moreover, writing $D_{\widetilde{H}}'\chi_{E_u}:=(D_{\widetilde{X}_1}\chi_{E_u},\dots,D_{\widetilde{X}_m}\chi_{E_u})$, then the following statements hold:
\begin{itemize}
	\item[(i)] $\pi_\#D_{\widetilde{X}_i}\chi_{E_u}=D_{X_i}u$ for any $i=1,\dots,m$;
	\item[(ii)] $\pi_\#\partial_t\chi_{E_u}=-\mathscr{L}^n$;
	\item[(iii)] $\pi_\#|D_{\widetilde{X}_i}\chi_{E_u}|=|D_{X_i}u|$ for any $i=1,\dots,m$;
	\item[(iv)] $\pi_\#|\partial_t\chi_{E_u}|=\mathscr{L}^n$;
	\item[(v)] $\pi_\#|D_{\widetilde{H}}'\chi_{E_u}|=|D_Hu|$.
	\item[(vi)] $\pi_\#|D_{\widetilde{H}}\chi_{E_u}|=|(D_Hu,-\mathscr{L}^n)|$.
	\end{itemize} 
\end{theorem}
\begin{proof}
Suppose first that $\chi_{E_u} \in BV_{H}(\Omega \times \mathbb{R})$ with respect to the family $\widetilde X$.  We  need to fix a sequence $(g_h)_h$ in $C_c^\infty(\mathbb{R})$ such that $g_h$ is even, $g_h\equiv1$ on $ [0,h]$, $g_h\equiv0$ on $[h+1,+\infty)$ and $\int_\mathbb{R}g_h(t)dt =2h+1$. Let $\varphi \in C_c^1(\Omega,\R^m)$ with $|\varphi|\leq 1$ be fixed. By the Dominated Convergence Theorem 
we have 
\[
\begin{aligned}
\int_{\Omega \times \mathbb{R}}\varphi(x) \cdot d(D_{\widetilde{H}}'\chi_{E_u})(x,t)&=\lim_{h\to+\infty}\int_{\Omega \times \mathbb{R}}g_h(t)\varphi(x)\cdot  d(D_{\widetilde{H}}'\chi_{E_u})(x,t)\\
	&=\displaystyle-\lim_{h\to+\infty}\int_{\Omega \times \mathbb{R}} \chi_{E_u}(x,t)g_h(t) \dive\varphi(x)d\mathscr{L}^{n+1}(x,t)\\
	&=-\lim_{h\to+\infty}\int_\Omega\left(\int_{-\infty}^{u(x)}g_h(t)dt\right)  \dive\varphi(x)d\mathscr{L}^n(x). 
\end{aligned}
\]
For every $z \in \mathbb R$ and  every $h \in \mathbb{N}$ we have
\[
\int_{-\infty}^z g_h(t)dt\leq |z|+h+\frac{1}{2}\quad\text{and}\quad
\lim_{h\to+\infty}\left(\int_{-\infty}^z g_h(t)dt-h-\frac{1}{2}\right)=z;
\]
using the fact that $\int_\Omega \dive\varphi(x)d\mathscr{L}^n(x)=0$, by the Dominated Convergence Theorem we obtain
\begin{equation}\label{primo}
\begin{split}
\int_{\Omega \times \mathbb{R}}\varphi(x)\cdot d(D_{\widetilde{H}}'\chi_{E_u})(x,t)
&=-\lim_{h\to+\infty}\int_\Omega \left(\int_{-\infty}^{u(x)}g_h(t)dt-h-\frac{1}{2}\right) \dive\varphi(x) d\mathscr{L}^n(x)\\
&=\displaystyle-\int_\Omega u(x)\dive\varphi(x)d\mathscr{L}^n(x)\\
&=\int_{\Omega }\varphi(x)\cdot d(D_{H} u)(x).
\end{split}
\end{equation}
In particular, $u\in BV_H(\Omega)$ and, for any open set $A\subset\Omega$,
\begin{equation}\label{eq:12.1}
\begin{split}
&|D_{H}u|(A)\leq |D_{\widetilde{H}}'\chi_{E_u}|(A\times\mathbb{R})\\
&|D_{X_i}u|(A)\leq |D_{\widetilde{X}_i}\chi_{E_u}|(A\times\mathbb{R})\quad\text{ for any }i=1,\dots,m.
\end{split}
\end{equation}
Before passing to the reverse implication we observe two facts. First, for any $\varphi\in C_c^1(\Omega)$ one has
\begin{equation}\label{secondo}
\begin{split}
\int_{\Omega\times\mathbb{R}}\varphi(x)d\left(\partial_t\chi_{E_u}\right)(x,t)
&=\lim_{h\to+\infty}\int_{\Omega\times\mathbb{R}}\varphi(x)g_h(t)d\left(\partial_t\chi_{E_u}\right)(x,t)\\
&=-\lim_{h\to+\infty}\int_{\Omega\times\mathbb{R}}\varphi(x)g_h'(t)\chi_{E_u}(x,t) d\mathscr{L}^{n+1}(x,t)\\
&=-\lim_{h\to+\infty}\int_{\Omega}\varphi(x)\left(\int_{-\infty}^{u(x)} g_h'(t)dt\right)d\mathscr{L}^n(x)\\
&=-\lim_{h\to+\infty}\int_\Omega\varphi(x) g_h(u(x))d\mathscr{L}^n(x)\\
&=-\int_\Omega\varphi d\mathscr{L}^n
\end{split}
\end{equation}
whence, for any open set $A\subset\Omega$,
\begin{equation}\label{eq:14.1}
\mathscr{L}^n(A) \leq |\partial_t\chi_{E_{u}}|(A\tR).
\end{equation}
Second, if  $\varphi\in C_c^1(\Omega,\R^{m+1})$ one has by \eqref{primo} and \eqref{secondo} 
\[
\int_{\Omega \times \mathbb{R}}\varphi(x)\cdot d(D_{\widetilde{H}}\chi_{E_u})(x,t) = \int_{\Omega }\varphi(x)\cdot d(D_{H} u,-\mathscr{L}^n)(x)
\]
which gives for any open set $A\subset\Omega$
\begin{equation}\label{eq:terzo}
|(D_Hu,-\mathscr{L}^n)|(A)\leq |D_{\widetilde{H}}\chi_{E_u}|(A\tR).
\end{equation}

Suppose now that $u\in BV_H(\Omega)$. Let $A\subset\Omega$ be open and let $\varphi \in C_c^1(A\times\mathbb{R})$ and $i=1,\dots,m$ be fixed. Let  $(u_h)_h$ be a sequence in $C^\infty(A)\cap BV_H(A)$ satisfying \eqref{eq:smoothapproxstrict} (with $A$ in place of $\Omega$); then
\begin{equation}\label{eq:blabla1}
\begin{split}
&\int_{A\times\mathbb{R}}\varphi\: d(D_{\widetilde X_i}\chi_{E_{u_h}})\\
&=-\int_{A\times\mathbb{R}} \chi_{E_{u_h}}(x,t)\widetilde X_i^*\varphi(x,t) d\mathscr{L}^{n+1}(x,t)\\
&=-\displaystyle\int_A\left(\int_{-\infty}^{u_h(x)}\sum_{j=1}^n\partial_{x_j}\left(a_{ij}(x)\varphi(x,t)\right)dt\right)d\mathscr{L}^n(x)\\
&=\displaystyle-\int_A\left(\sum_{j=1}^n\partial_{x_j}\int_{-\infty}^{u_h(x)}a_{ij}(x)\varphi(x,t)dt -\sum_{j=1}^n a_{ij}(x)\varphi(x,u_h(x))\partial_{x_j}u_h(x)\right)d\mathscr{L}^n(x)\\
&=\displaystyle\int_A \varphi(x,u_h(x))X_i u_h(x)d\mathscr{L}^n(x),
\end{split}
\end{equation}
where we used the fact that $x\mapsto a_{ij}(x)\int_{-\infty}^{u_h(x)}\varphi(x,t)dt$ is in $C_c^1(A)$. In a similar way
\begin{equation}\label{eq:blabla2}
\begin{split}
\int_{A\times\mathbb{R}}\varphi\:d\big(\partial_t\chi_{E_{u_h}}\big)& =-\int_A\left(\int_{-\infty}^{u_h(x)}\partial_t\varphi(x,t)dt\right)d\mathscr{L}^n(x)\\
&=\displaystyle-\int_A \varphi(x,u_h(x))d\mathscr{L}^n(x)\
\end{split}
\end{equation}
Formulas \eqref{eq:blabla1} and \eqref{eq:blabla2} imply that for any $\varphi \in C_c^1(A\times\mathbb{R},\R^{m+1})$
\[
\int_{A\times\mathbb{R}}\varphi\cdot d(D_{\widetilde H}\chi_{E_{u_h}}) = \int_A \varphi(x,u_h(x))\cdot d(D_Hu_h,-\mathscr{L}^n)(x)
\]
Since  $\chi_{E_{u_h}}\rightarrow\chi_{E_u}$ in $L^1(A\tR)$ we obtain 
\begin{equation}\label{eq:17}
\begin{split}
|D_{\widetilde H}\chi_{E_u}|(A\times \mathbb{R})\leq & \liminf_{h\to+\infty}|D_{\widetilde H}\chi_{E_{u_h}}|(A\times\mathbb{R} )\leq \lim_{h\to+\infty}|(D_{H}u_h,-\mathscr{L}^n)|(A)\\
= & |(D_{H}u,-\mathscr{L}^n)|(A)<+\infty,
\end{split}
\end{equation}
which proves that $\chi_{E_u} \in BV_{\widetilde{H}}(\Omega\times\mathbb{R})$, as desired. Notice that, using the lower semicontinuity in a similar way,  one also gets
\begin{equation}\label{eq:18}
\begin{split}
&|D_{\widetilde H}'\chi_{E_u}|(A\times \mathbb{R})\leq |D_{H}u|(A)\\
&|D_{\widetilde X_i}\chi_{E_u}|(A\times \mathbb{R})\leq |D_{X_i}u|(A)\quad\text{for any }i=1,\dots,m\\
&|\partial_t\chi_{E_{u}}|(A\tR)\leq \mathscr{L}^n(A)<+\infty.
\end{split}
\end{equation}

Eventually, statements (i) and (ii) follow from \eqref{primo} and  \eqref{secondo}, while statements (iii)--(vi) are consequences of formulas  \eqref{eq:12.1},   \eqref{eq:14.1}, \eqref{eq:terzo}, \eqref{eq:17} and \eqref{eq:18}. 
\end{proof}

Let us introduce some further notation. For $u\in BV_{H,loc}(\Omega)$ we decompose its  distributional horizontal derivatives as $D_Hu=D_H^au+D_H^su$, where $D_H^au$ is absolutely continuous with respect to $\mathscr{L}^n$ and $D_H^su$ is singular with respect to $\mathscr{L}^n$. We also write $D_H^au=Xu\,\mathscr{L}^n$ for some function $Xu\in L^1_{loc}(\Omega,\R^m)$.

We  also consider the polar decomposition $D_Hu=\si_u|D_Hu|$, where $\si_u:\Omega\to \mathbb S^{m-1}$ is a $|D_Hu|$-measurable function. In case $u=\chi_E$ is the characteristic function of a set $E\subset\Omega\tR$ of locally finite $\widetilde H$-perimeter in  $\Omega\tR$ we write $D_{\widetilde H}\chi_E=\nu_E|D_{\widetilde H}\chi_E|$ for some Borel function $\nu_E=((\nu_{E})_{1},\dots,(\nu_{E})_{m+1})$ called  {\em horizontal inner normal} to $E$. 

The following result is basically a consequence of Theorem \ref{subgraph}.

\begin{theorem}\label{teo:normalevettorepolare}
Let $u\in BV_H(\Omega)$ and define 
\[
\begin{split}
& S:=\left\{(x,t)\in \Omega\times \mathbb R:(\nu_{E_u})_{m+1}(x,t)=0\right\}\\
& T:=\left\{(x,t)\in \Omega\times \mathbb R:(\nu_{E_u})_{m+1}(x,t)\neq0\right\}.
\end{split}
\]
Then,  the following identities hold 
\begin{eqnarray}
&& \!\!\!\nu_{E_u}(x,t)=(\sigma_u(x),0)\quad\text{for $|D_{\widetilde H}\chi_{E_u}|$-a.e. $(x,t)\in S$;}\label{eq:(a)}\\
&& \!\!\!\nu_{E_u}(x,t)=\frac{(Xu(x),-1)}{\sqrt{1+|Xu(x)|^2}} \quad\text{for $|D_{\widetilde H}\chi_{E_u}|$-a.e. $(x,t)\in T $;}\label{eq:(b)}\\
&& \!\!\!\pi_\#(D_{\widetilde H}\chi_{E_u}\res S)=(D_H^su,0);\label{eq:(c1)}\\
&& \!\!\!\pi_\#(D_{\widetilde{H}}\chi_{E_u}\res T)=(D_H^au,-\mathscr L^n).\label{eq:(c2)}
\end{eqnarray}

\end{theorem}
\begin{proof}
Thanks to Theorem \ref{subgraph} (vi)   we can disintegrate the measure $|D_{\widetilde H}\chi_{E_u}|$ with respect to $|(D_Hu,-\mathscr L^n)|$ (see e.g. \cite[Theorem 2.28]{AFP}): for every $x\in \Omega$ there exists a probability measure $\mu_x$ on $\R$ such that for every Borel function $g\in L^1(\Omega\times \mathbb R,|D_{\widetilde H}\chi_{E_u}|)$
\[
	\int_{\Omega\times \mathbb R} g(x,t)d|D_{\widetilde H}\chi_{E_u}|(x,t)=\int_\Omega\left(\int_\mathbb R g(x,t) d\mu_x(t)\right)d|(D_Hu,-\mathscr L^n)|(x). 
\]
It follows that for any Borel function $\varphi:\Omega\to\R$ 
\begin{equation}\label{eq:piove}
	\begin{aligned}
	\int_\Omega \varphi(x)d(D_Hu,-\mathscr{L}^n)(x)&=\int_\Omega \varphi(x)d\pi_\#(\nu_{E_u}|D_{\widetilde H}\chi_{E_u}|)(x)\\
	&=\int_{\Omega\times \mathbb R} \varphi(x)\nu_{E_u}(x,t)d|D_{\widetilde H}\chi_{E_u}|(x,t)\\
	&=\int_\Omega \varphi(x)\left(\int_\mathbb R\nu_{E_u}(x,t)d\mu_x(u)\right)d |(D_Hu,-\mathscr{L}^n)| (x).
	\end{aligned}
\end{equation}
Since $D_H^au$ and $D_H^su$ are mutually singular we have
\[
|(D_Hu,-\mathscr{L}^n)| = |(D_H^au,-\mathscr{L}^n)| + |(D_H^su,0)| = \sqrt{1+|X u|^2}\mathscr L^n +|D_H^su|
\]
and \eqref{eq:piove} gives
\begin{align}\label{eqdisintegration}
	&\int_\Omega\varphi\: d\Big((Xu,-1)\mathscr L^n+(\sigma_u,0)|D_H^su|\Big)\\
	=&\int_\Omega\varphi(x)\left(\int_\mathbb R \nu_{E_u}(x,t)d\mu_x(t)\right)d\left(\sqrt{1+|X u|^2}\mathscr L^n +|D_H^su|\right)(x).
\end{align}
Denote by $I$ a subset of $\Omega$ such that $\mathscr L^n(I)=0$ and $|D_H^su|(\Omega\setminus I)=0$.
	Considering  Borel test functions $\varphi$  such that $\varphi= 0$ in $\Omega\setminus I$, we deduce that for $|D_H^su|$-a.e. $x\in I$ one has
	\[
	(\sigma_u(x),0)=\int_\mathbb R \nu_{E_u}(x,t)d\mu_x(t).
	\]
Taking on both sides the scalar product  with $(\sigma_u(x),0)$ we get
	\[
	\left\langle(\sigma_u(x),0),\int_\mathbb R \nu_{E_u}(x,t)d\mu_x(t)\right\rangle=1,
	\]
	and, since $\mu_x(\mathbb R)=1$ and (for $|(D_Hu,-\mathscr L^n)|$-a.e. $x\in\Omega$) $|\nu_{E_u}(x,t)|=1$ for $\mu_x$-a.e. $t$,  we deduce that 
	\[
	\nu_{E_u}(x,t)=(\sigma_u(x),0)\qquad\text{for $|D_H^su|$-a.e. $x\in I$ and  $\mu_x$-a.e. $t\in \mathbb R$,}
	\]
i.e.,
	\begin{equation}\label{eq:concl1}
	\nu_{E_u}(x,t)=(\sigma_u(x),0)\qquad\text{for $|D_{\widetilde H}\chi_{E_u}|$-a.e. $(x,t)\in I\tR$.}
	\end{equation}	

	Taking into account again \eqref{eqdisintegration} and letting $\varphi$ be such that  $\varphi=0$ on $I$ we instead obtain
	\begin{align*}
	&\int_\Omega\varphi\:\frac{(Xu,-1)}{\sqrt{1+|X u|^2}}\sqrt{1+|X u|^2} d\mathscr L^n\\
	= & \int_{\Omega} \varphi(x) \left(\int_\mathbb R \nu_{E_u}(x,t)d\mu_x(t)\right)\sqrt{1+|Xu(x)|^2}d\mathscr L^n(x)
	\end{align*}
	Consequently, for $\mathscr L^n$-a.e. $x\in \Omega\setminus I$ we have
	\[
	\int_\mathbb R\nu_{E_u}(x,t)d\mu_x(t)=\frac{(Xu(x),-1)}{\sqrt{1+|Xu(x)|^2}}.
	\]
Reasoning as before we deduce that   
	\[
	\nu_{E_u}(x,t)=\frac{(Xu(x),-1)}{\sqrt{1+|Xu(x)|^2}}\quad\text{for $\mathscr L^n$-a.e. $x \in \Omega \setminus I$ and  $\mu_x$-a.e. $t\in \mathbb R$,}
	\]
or equivalently
	\begin{equation}\label{eq:concl2}
	\nu_{E_u}(x,t)=\frac{(Xu(x),-1)}{\sqrt{1+|Xu(x)|^2}}\quad\text{for $|D_{\widetilde H}\chi_{E_u}|$-a.e. $(x,t)\in (\Omega\setminus I)\tR$.}
	\end{equation}
Formula \eqref{eq:concl1} implies that  $|D_{\widetilde H}\chi_{E_u}|$-a.e. $(x,t) \in I\times \mathbb R$ belongs to $S$ and that $|D_{\widetilde H}\chi_{E_u}|$-a.e. $(x,t)\in T $ belongs to $(\Omega\setminus I) \times \mathbb R$. Similarly,  \eqref{eq:concl2} says that  $|D_{\widetilde H}\chi_{E_u}|$-a.e. $(x,t)\in (\Omega\setminus I) \times \mathbb R$ belongs to $T$ and that $|D_{\widetilde H}\chi_{E_u}|$-a.e. $(x,t) \in S$ belongs to $I\times \mathbb R$. Since $S$ and $T$ are disjoint, this is enough to conclude \eqref{eq:(a)} and \eqref{eq:(b)}. Statement \eqref{eq:(c1)}  now easily follows because
\[
\pi_\#(D_{\widetilde H}\chi_{E_u}\res S) = \pi_\#(\nu_{E_u}|D_{\widetilde H}\chi_{E_u}|\res (I\tR))=  (\si_u,0)|(D_Hu,-\mathscr L^n)|\res I = (D_H^su,0)
\]
Similarly, one has
\[
\begin{split}
\pi_\#(D_{\widetilde H}\chi_{E_u}\res T) = & \pi_\#(\nu_{E_u}|D_{\widetilde H}\chi_{E_u}|\res ((\Omega\setminus I)\tR))\\
= &  \frac{(Xu,-1)}{\sqrt{1+|Xu|^2}}|(D_Hu,-\mathscr L^n)|\res (\Omega\setminus I) = (Xu,-1)\mathscr L^n,
\end{split}
\]
which gives \eqref{eq:(c2)}.
\end{proof}

\section{The rank-one theorem for \texorpdfstring{$BV_H$}{BVH} functions in Carnot groups}
We now use the results of the previous section in the setting of a Carnot group $\G$. We utilize the notation of Section \ref{sec:preliminari}; in particular, we  identify $\G\equiv\R^n$ by exponential coordinates and a left-invariant  basis $X_1,\dots,X_m$ of $\galg_1$ is fixed.  The vector fields $\widetilde X_1,\dots,\widetilde X_{m+1}$ on $\G\tR$ are defined as in the previous section; notice that they form a basis of the first layer of the Lie algebra of $\G\tR$. The homogeneous dimension of $\G\tR$ is $Q+1$.

A set $R\subset\G$ is {\em $H$-rectifiable} if $\mathcal H^{Q-1}(R)<\infty$ and there exists a (finite or countable) family $(\Si_i)_i$ of $C^1_H$ hypersurfaces in $\G$ such that
\[
\mathcal H^{Q-1}\Big(R\setminus\bigcup_i \Si_i\Big)=0.
\]
We  define the {\em horizontal normal} $\nu_R$ to $R$ as
\[
\nu_R(x):=\nu_{\Si_i}(x)\qquad\text{if }x\in R\cap \Si_i\setminus\cup_{j<i} \Si_j.
\]
The normal $\nu_R$ is  well-defined (up to sign) $\mathcal H^{Q-1}$-a.e. on $R$.\footnote{The key property to prove this assertion is that the set of points where  two $C^1_H$ hypersurfaces intersect transversally is  $\mathcal H^{Q-1}$-negligible: this fact holds true in  any {\em equiregular Carnot-Carath\'eodory space}, see e.g. \cite{Don}. Actually, in view of Theorem \ref{teo:rank1} we could restrict to the setting of  Carnot groups satisfying property $\prop_2$, where the claim follows from Theorem \ref{teo:intersezione}.}


\begin{definition}\label{def:Rprop}
We say that a Carnot group $\G$ satisfies property $\Rprop$ if the following holds. For any bounded open set $\Omega\subset\G$ and any $u\in BV_{H}(\Omega)$, the distributional $\widetilde X$-derivatives  $D_{\widetilde H}\chi_{E_u}$ of the characteristic function of the subgraph $E_u$ of $u$ can be represented as
\begin{equation}\label{eq:rettifperim}
D_{\widetilde H}\chi_{E_u} = \nu_{\partial_H^*E_u} \theta \Shaus^{Q}\res \partial_H^*E_u
\end{equation}
for  some $H$-rectifiable set $\partial_H^*E_u$ in $\Omega\tR$  and some positive density $\theta\in L^1(\partial_H^*E_u,\Shaus^{Q})$.   We  call $\partial_H^*E_u$ the {\em $H$-reduced boundary of $E_u$}.
\end{definition}

Notice that, in Definition \ref{def:Rprop}, the measure $D_{\widetilde H}\chi_{E_u}$ has finite total variation by Theorem \ref{subgraph}.

\begin{remark}
In view of Theorem \ref{teo:subgraphsemplificato}, for the validity of property $\Rprop$ in $\G$ it is enough that a rectifiability theorem holds for sets with  finite $H$-perimeter  in $\G\tR$; namely, it suffices that any set $E$ with  finite $H$-perimeter in $\G\tR$ satisfies $D_{\widetilde H}\chi_E=\nu_{\partial_H^*E} \theta \Shaus^{Q}\res \partial_H^*E$ for some   $H$-rectifiable set $\partial_H^*E$  and some  positive density $\theta\in L^1(\partial_H^*E,\Shaus^{Q})$. We conjecture that this, in turn, is equivalent to the validity of a rectifiability theorem for sets with  finite $H$-perimeter  in $\G$; in particular, we conjecture that property $\Rprop$ is equivalent to the rectifiability theorem in $\G$.
\end{remark}

\begin{remark}\label{rem:step2Rprop}
If $\G$ is a Carnot group of step 2, then $\G$ satisfies property $\Rprop$: this follows from the fact that  $\G\tR$ is also a step 2 Carnot group and that the rectifiability theorem holds  in any step 2 Carnot group, see  \cite{FSSCstep2}.
\end{remark}

\begin{remark}\label{rem:paroleparoleparole}
If \eqref{eq:rettifperim} holds, then
\[
|D_{\widetilde H}\chi_{E_u}|=\theta \Shaus^{Q}\res \partial_H^*E_u\quad\text{and}\quad\nu_{E_u}=\nu_{\partial_H^*E_u}\ \Shaus^{Q}\text{-a.e. on }\partial_H^*E_u.
\]
\end{remark}

\begin{proof}[Proof of Theorem \ref{teo:rank1}]
Without loss of generality one can assume that  $u=(u_1,\dots,u_d)\in BV_H(\Omega,\R^d)$. It is not restrictive to assume that $\Omega$ is bounded. For any $i=1,\dots,d$ we write $D_H^su_i=\sigma_i|D_H^su_i|$ for a $|D_H^su_i|$-measurable map $\sigma_i:\Omega\to {\mathbb S}^{m-1}$; notice that, using the notation of Section \ref{sec:sottografici}, the equality $\si_i=\si_{u_i}$ holds  $|D^su_i|$-almost everywhere. We also let $E_i:=\{(x,t)\in\Omega\times\R:\,t<u_i(x)\}$ be the subgraph of $u_i$, that has finite $H$-perimeter in $\Omega\times\R$ by Theorem \ref{subgraph}. Denoting by $\partial_H^* E_i$ the $H$-reduced boundary of $E_i$ and writing $\nu_i=\nu_{E_i}$ for the measure theoretic inner normal to $E_i$, we have by Theorem \ref{teo:normalevettorepolare} and Remark \ref{rem:paroleparoleparole} that
\[
|D_H^su_i|=\pi_\#(\theta_i{\Shaus}^{Q}\res S_i)\quad\text{for some positive }\theta_i\in L^1(\partial_H^* E_i,{\Shaus}^{Q}),
\]
where $S_i:=\left\{p\in\partial_H^\ast E_i:\left(\nu_i(p)\right)_{m+1}=0\right\}$ 
and $\pi_\#$ denotes push-forward of measures through the projection $\pi$ defined by $\G\tR\ni(x,t)\mapsto x\in\G$. By rectifiability, we can assume that $\partial_H^\ast E_i$ is contained in the union $\cup_{\ell\in\N} \Sigma^i_\ell$ of $C^1_H$ hypersurfaces $\Sigma^i_\ell$ in $\G\tR$. 

Using Theorem \ref{teo:normalevettorepolare}, Remark \ref{rem:paroleparoleparole} and Lemma \ref{lem:trascurabilita} the following properties hold for $\Shaus^Q$-a.e. $p\in S_1\cup\dots\cup S_d$:
\begin{align}
& \text{if $p\in S_i$, then }\nu_i(p) = (\sigma_i(\pi(p)),0)  \label{2.1}\\
& \text{if }p\in\Sigma^i_\ell, \text{ then }  \nu_{i}(p) = \pm\nu_{\Sigma^i_\ell}(p) \label{2.2}\\
& \text{if }p\in \Sigma^i_\ell\text{ and }\exists\:  q\in S_j\cap\Sigma^j_k\cap\pi^{-1}(\pi(p)),\text{ then }  \nu_{\Sigma^i_\ell}(p)=\pm \nu_{\Sigma^j_k}(q). \label{2.3}
\end{align}
Up to modifying each $S_i$ on a $\Shaus^Q$-negligible set and each $\sigma_i$ on a $|D_H^su_i|$-negligible set, we can assume that  \eqref{2.1}, \eqref{2.2} and \eqref{2.3} hold for any $p\in S_1\cup\dots\cup S_d$ and that, for any $i=1,\dots,d$, $\sigma_i=0$ on $\Omega\setminus\pi(S_i)$.

Since $D_H^su=(\sigma_1|D_H^su_1|,\dots,\sigma_d|D_H^su_d|)$ and $|D_H^su|$ is 
concentrated on $\pi(S_1)\cup\dots\cup \pi(S_m)$, it is enough to prove that the matrix-valued function $(\sigma_1,\dots,\sigma_m)$ has rank 1 on $\pi(S_1)\cup\dots\cup \pi(S_m)$. This  follows if we prove that the implication
\[
i, j\in\{1,\ldots,d\},\ i\neq j,\ x\in\pi(S_i)\ \Longrightarrow\ \sigma_j(x)\in\{0,\sigma_i(x),-\sigma_i(x)\}
\]
holds. If $i,j,x$ are as above and $x\notin \pi(S_j)$, then $\sigma_j(x)=0$. Otherwise, $x\in \pi(S_i)\cap \pi(S_j)$, i.e., there exist $p\in S_i$ and $\ell\in\N$ such that $\pi(p)=x$ and $\sigma_i(x)=\pm\nu_{\Sigma^i_\ell}(p)$ and there exist $q\in S_j$ and $k\in\N$ such that $\pi(q)=x$ and $\sigma_j(x)=\pm\nu_{\Sigma^j_k}(p)$. By \eqref{2.3} we obtain $\sigma_j(x)=\pm \sigma_i(x)$, as wished.
\end{proof}

\begin{remark}\label{rem:rangounoH2}
As an easy consequence of Remark \ref{rem:esempiPk} and Remark \ref{rem:step2Rprop}, Theorem \ref{teo:rank1} holds for the Heisenberg group $\H^n$ provided $n\geq 2$. This result does not directly follow from \cite{DPR}, as we now briefly explain using the notation of Example \ref{ex:introHn} and restricting for simplicity to $n=2$, the general case $n\geq 2$ being a straightforward generalization.

Let $u\in BV_{H}(\Omega,\R^m)$ for some open set $\Omega\subset\H^2$. It can be easily seen that the matrix-valued measure $(\mu_1,\mu_2,\mu_3,\mu_{4}):= D_Hu=(X_1u,X_2u,Y_1u,Y_2u)$ satisfies the equations 
\[
\mathscr A\mu:=\left(
\begin{array}{l}
X_1\mu_2-X_2\mu_1\\
Y_1\mu_4-Y_2\mu_3\\
X_1\mu_4-Y_2\mu_1\\
Y_1\mu_2-X_2\mu_3\\
X_1\mu_3-Y_1\mu_1  +  Y_2\mu_2    -X_2\mu_4
\end{array}
\right)=0
\] 
in the sense of distributions. Write the first-order differential operator $\mathscr A$ (the {\em horizontal curl} in $\H^2$, see \cite[Example 3.12]{BF}) in the form
\[
\mathscr{A}=A_1\partial_{x_1} + A_2\partial_{x_2} + A_3\partial_{y_1} + A_4\partial_{y_2}+ A_{5}\partial_{t}
\]
for suitable $A_j=A_j(x,y,t)$ and consider the  \emph{wave cone}  $\Lambda_{\mathscr{A}}(x,y,t)$ (see \cite{DPR}) associated with $\mathscr{A}$ 
\[
\Lambda_{\mathscr{A}}(x,y,t):=\bigcup_{\xi\in\R^{5}\setminus\{ 0\}}
\ker \mathbb{A}_{x,y,t}(\xi),
\qquad\text{where } \mathbb{A}_{x,y,t}(\xi):= 2\pi i
\displaystyle\sum_{j=1}^{5} A_j(x,y,t) \xi_j.
\] 
One can  readily check that
\[
\mathbb{A}_{x,y,t}(\xi)=0\quad\text{for }\xi:=(\tfrac y2,-\tfrac x2, 1)\in\R^{5}\setminus\{ 0\},
\]
i.e., the wave cone $\Lambda_{\mathscr{A}}(x,y,t)$ is the full space for any $(x,y,t)\in\H^2$. In particular, \cite[Theorem 1.1]{DPR} gives no information on the polar decomposition of  $D_H^su$.
\end{remark}

\begin{remark}\label{rem:rangounoH1}
The rank-one property for $BV$ functions in the first Heisenberg group remains a very interesting open question, since it does not follow either from Theorem \ref{teo:rank1} (because property $\prop_2$ fails for $\H^1$) or from \cite[Theorem 1.1]{DPR}, as we now explain.

Let $u\in BV_{H}(\Omega,\R^m)$ for some open set $\Omega\subset\H^1$; we use again the  notation of Example \ref{ex:introHn} and we set $p=(x,y,t)\in\H^1\equiv\R^3$. One can check  that $(\mu_1,\mu_2):=D_Hu=(Xu,Yu)$ satisfies
\[
\mathscr A\mu:=\left(
\begin{array}{l}
Y\!X\mu_1-2XY\mu_1+XX\mu_2\\
YY\mu_1-2Y\!X\mu_2+XY\mu_2,
\end{array}
\right)=0
\]
in the sense of distributions. Now  $\mathscr A$ (the horizontal curl in $\H^1$, see \cite[Example 3.11]{BF})  is a second-order differential operator that one can write as
\[
\mathscr{A}=\sum_{|\alpha|=2} A_\alpha(p)\partial^{\alpha},
\]
where $\alpha \in \mathbb{N}^3$ is a multi-index and $\partial^{\alpha}=\partial_{x}^{\alpha_1}\partial_{y}^{\alpha_2}\partial_{t}^{\alpha_3}$. As before, one can  define the wave cone 
\[
\Lambda_{\mathscr{A}}(p)=\bigcup_{\xi\in\R^3\setminus\{0\}}\ker \mathbb{A}_p(\xi),
\qquad\text{where }
\mathbb{A}_p(\xi)=\displaystyle(2\pi i)^2\sum_{|\alpha|=2} A_\alpha(p) \xi^\alpha.
\] 
Again, one has 
\[
\mathbb{A}_{p}(\xi)=0\qquad\text{for }\xi:=(\tfrac y2,-\tfrac x2, 1)\in\R^{3}\setminus\{ 0\}
\]
and the wave cone $\Lambda_{\mathscr{A}}(x,y,t)$ is the full space.
\end{remark}

\appendix
\renewcommand{\theequation}{\thesection.\arabic{equation}}
\section{Intersection of regular hypersurfaces vs. intrinsic Lipschitz graphs}\label{sec:appendice}
\subsection{Intrinsic Lipschitz graphs}\label{subsec:intLipgr}
We follow  \cite{FSJGA}. Let $\W,\H$ be homogeneous (i.e., invariant under dilations)  complementary subgroups of $\G$, i.e., such that $\W\cap\H=\{0\}$ and $\G=\W\H$. In particular, for any $x\in\G$ there exist unique $x_\W\in\W$ and $x_\H\in\H$ such that $x=x_\W x_\H$. Recall (see e.g. \cite[Remark 2.3]{FSJGA}) that any homogeneous subgroup $\W$ is stratified, that is, its Lie algebra $\mathfrak w$ is a subalgebra of $\galg$ and $\mathfrak w=\mathfrak w_1\oplus\dots\oplus\mathfrak w_s$ where $\mathfrak w_i=\mathfrak w\cap\galg_i$. Moreover, the metric (Hausdorff) dimension of $\W$ is $Q_\W:=\sum_{i=1}^si\:\dim \mathfrak w_i$.

The {\em intrinsic graph} of a function $\f: \W\to\H$ is defined by
\[
\mathrm{gr\ }\f:=\{w\f(w):w\in\W\}.
\]
We introduce the homogeneous cones $C_{\W,\H}(x,\al)$ of center $x\in \G$ and aperture $\al>0$ as
\[
C_{\W,\H}(x,\al):=xC_{\W,\H}(0,\al)\quad\text{where}\quad C_{\W,\H}(0,\al):=\{y\in\G:\|x_\W\|\leq\al\| x_\H\|\}.
\]

\begin{definition}\label{def:intrLip}
A function $\f: \W\to\H$ is {\em intrinsic Lipschitz} if there exists $\al>0$ such that
\[
\forall\:x\in \mathrm{gr\ }\f\qquad \mathrm{gr\ }\f\cap C_{\W,\H}(x,\al)=\{x\}.
\]
We say that $S\subset\G$  is an {\em intrinsic Lipschitz graph} if there exists an intrinsic Lipschitz map $\f:\W\to\H$ such that $S=\mathrm{gr\ }\f$.
\end{definition}

\begin{remark}\label{rem:altriconi}
We will later use the following equivalent definition of intrinsic Lipschitz continuity: $\f:\W\to\H$ is intrinsic Lipschitz if and only if there exists $\be>0$ such that
\[
\forall\:x\in \mathrm{gr\ }\f\qquad \mathrm{gr\ }\f\cap D(x,\H,\be)=\{x\}
\]
where the homogeneous cone $D(x,\H,\be)$ is defined by
\[
D(x,\H,\be):=xD(\H,\be)
\quad\text{and}\quad
D(\H,\be):=\bigcup_{h\in\H} \overline{B(h,\be d(h,0))}.
\]
Indeed, it is enough to observe that, for any $\al>0$ and $\be>0$, there exist $\be_\al>0$ and $\al_\be>0$ such that
\[
C_{\W,\H}(x,\al)\supset D(\H,\be_\al)\quad\text{and}\quad D(\H,\be)\supset C_{\W,\H}(x,\al_\be).
\]
This, in turn, is a consequence of a homogeneity argument based on the following fact: if $S:=\{x\in\G:\|x\|=1\}$ and
\[
A_\al:=S\cap\mathrm{int}(C_{\W,\H}(x,\al)),\qquad B_\be:=S\cap\mathrm{int}(D(\H,\be)),
\]
then   $\{A_\al\}_{\al>0}$ and $\{B_\be\}_{\be>0}$ are monotone families of (relatively) open subsets of $S$ such that the intersection
\[
\bigcap_{\al>0}A_\al=\bigcap_{\be>0}B_\be= \H\cap S
\]
is a compact set.
\end{remark}

The following result will be used in the proof of Theorem \ref{teo:intersezione}.

\begin{theorem}[{\cite[Theorem 3.9]{FSJGA}}]\label{teo:FSJGAmisurafinita}
Let $\W,\H$ be homogeneous complementary subgroups of $\G$, let $\f:\W\to\H$ be intrinsic Lipschitz and let $\al>0$ be as in Definition \ref{def:intrLip}. Then there exists a positive $C=C(\W,\H,\al)$ such that
\[
\frac 1C r^{Q_\W}\leq \mathcal H^{Q_\W}(\mathrm{gr\ }\f\cap B(x,r))\leq C r^{Q_\W}\qquad\forall\:x\in\mathrm{gr\ }\f,r>0.
\]
\end{theorem}

\subsection{Transversal intersections  of  \texorpdfstring{$C^1_H$}{C1H} hypersurfaces are    intrinsic Lipschitz graphs}
The aim of this section is proving  Theorem \ref{teo:interse'grafLip}, due to V.~Magnani \cite{MagTowards}, for which we  need the  preparatory Lemma \ref{lem:estensioneC1H}. Actually, its use could be avoided by utilizing a local version of Theorem \ref{teo:FSJGAmisurafinita} which, even though not explicitly stated there, would easily follow adapting the techniques of \cite{FSJGA}. We note however that Lemma \ref{lem:estensioneC1H}, and \eqref{eq:antilocal} in particular, provides also a proof of \eqref{eq:blowup}. 

\begin{lemma}\label{lem:estensioneC1H}
Let $\Omega\subset\G$ be open, $f\in C^1_H(\Omega)$, $\bar x\in \Omega$ and let $A:=\nabla_Hf(\bar x)$. Then, for any $\ep>0$ there exist an open set $U\subset\Omega$ with $\bar x\in U$ and a function $g\in C^1_H(\G)$ such that
\begin{itemize}
\item[(i)] $g=f$ on $U$;
\item[(ii)] $| \nabla_Hg-A|<\ep$ on $\G$.
\end{itemize}
\end{lemma}
\begin{proof}
Without loss of generality we can assume that $\bar x=0$. We preliminarily fix a smooth function $\chi:\G\to[0,1]$ such that $\chi\equiv 1$ on $B_1$ and $\chi\equiv 0$ on $\G\setminus B_2$. For any $r>0$, the functions $\chi_r:=\chi\circ\de_{1/r}$ satisfy
\[
0\leq\chi_r\leq1,\quad\chi\equiv 1 \text{ on }B_r,\quad\chi\equiv 0 \text{ on }\G\setminus B_{2r},\quad|\nabla_H\chi_r|\leq\frac Cr
\]
for some positive $C$ independent of $r$. 

Let $\ep>0$ be fixed. We fix $r>0$ such that $| \nabla_Hf-A|<\ep$ on $B_{2r}$. With this choice, setting $\la(x):=A_1x_1+\dots+A_mx_m$ (where $x$ is represented in exponential coordinates) we prove that
\begin{equation}\label{eq:antilocal}
|f(x)-\la(x)|<2\ep r\quad\text{ for any }x\in B_{2r}.
\end{equation}
Indeed, for any $x\in B_{2r}$ there exists a horizontal curve $\gamma:[0,1]\to\G$ such that  $\gamma(0)=0$, $\gamma(1)=x$ and $L(\gamma)<2r$. By definition, there exists $h\in L^\infty([0,1],\R^m)$ such that
\[
\dot \gamma (t)= \sum_{i=1}^m h_i(t) X_i (\gamma(t))\quad\text{ for a.e. }t\in[0,1].
\]
Moreover, for any $i=1,\dots,m$ we have $\int_0^1h_i=x_i$, because in exponential coordinates one has $X_i(x)=\partial_{x_i}+\sum_{\ell>m+1}a_{i\ell}\partial_{x_\ell}$ (see e.g. \cite{Vtesi}). It follows that
\begin{align*}
|f(x)-\la(x)| = & \left| \int_0^1\sum_{i=1}^m h_i(t) X_if (\gamma(t))dt-\int_0^1\sum_{i=1}^m A_i h_i(t)dt\right|\\
\leq & \int_0^1|h(t)|\:\|\nabla_Hf (\gamma(t))- A\| dt\\
< & 2\ep r.
\end{align*}
We  now define $g:=\chi_r f + (1-\chi_r)\la$; statement (i) is readily checked, while for (ii)
\begin{align*}
| \nabla_Hg-A| =\:& |\chi_r \nabla_H f + (1-\chi_r)A+(f-\la)\nabla_H\chi_r-A|\\
\leq\:& \chi_r |\nabla_H f-A| + |f-\la||\nabla_H\chi_r|\\
\leq\:& \ep+2C\ep.
\end{align*}
The proof is then accomplished.
\end{proof}

We can now prove the main result of this section. Since property $\prop_1$ holds in any Carnot group, when $k=1$ Theorem \ref{teo:interse'grafLip}   states in particular  that hypersurfaces of class $C^1_H$ in a Carnot group $\G$ are locally intrinsic Lipschitz graphs of codimension 1.

\begin{theorem}[{\cite[Theorem 1.4]{MagTowards}}]\label{teo:interse'grafLip}
Let $\G$ be a Carnot group of rank $m$ and let $\Sigma_1,\dots,\Sigma_k$, $k\leq m$, be hypersurfaces of class $C^1_H$ with horizontal normals $\nu_1,\dots,\nu_k$; let $x\in \Si:=\Si_1\cap\dots\cap\Si_k$ be such that $\nu_1(x),\dots,\nu_k(x)$ are linearly independent. Consider the vertical plane $\W:=T_x\Si_1\cap\dots\cap T_x\Si_k$ of codimension $k$ and assume that there exists a complementary homogeneous horizontal subgroup $\H$ such that $\G=\W\H$. Then, there exists an open neighborhood $U$ of $x$ and an intrinsic Lipschitz $\f:\W\to\H$ such that
\[
\Si\cap U=\mathrm{gr\ }\f\cap U.
\]
\end{theorem}
\begin{proof}
We work in exponential coordinates associated with an adapted basis $X_1,\dots,X_n$ of $\galg$ such that
\[
\H=\exp(\text{span }\{X_1,\dots,X_k\}),\qquad \W=\exp((\text{span }\{X_{k+1},\dots,X_s\})\oplus\galg_2\oplus\dots\oplus\galg_s).
\]
By definition we can find an open neighborhood $U$ of $x$ and $f=(f_1,\dots,f_k)\in C^1_H(U,\R^k)$ such that $\Si\cap U=\{x\in U:f(x)=0\}\cap U$ and the $m\times k$ matrix-valued function $ \nabla_H f$ has  rank $k$ in $U$. Actually, by our choice of the basis the $k\times k$ minor $M:=(X_1f(x),\dots, X_kf(x))$ has rank $k$.

Let $\ep$ be a positive number, to be fixed later and only depending on $M$. By Lemma \ref{lem:estensioneC1H}, possibly restricting $U$ we can assume that $f$ is defined on the whole $\G$, that $f\in C^1_H(\G,\R^k)$ and $|\nabla_Hf-\nabla_Hf(x)|<\ep$; in particular,
\[
|(X_1f,\dots, X_kf)-M|<\ep\quad\text{on }\G.
\]
It is enough to prove that the level set $R:=\{x\in\G:f(x)=0\}$ is an intrinsic Lipschitz graph. We divide the proof of this claim into two steps.

{\em Step 1: $R$ is the intrinsic graph of some $\f:\W\to\H$.} It is enough to show that, for any $w\in\W$, there exists a unique $h\in\H$ such that $f(wh)=0$; in particular, this  allows  to define the map $\f$ by $\f(w):=h$. 

The map $(h_1,\dots,h_k)\longleftrightarrow\exp(h_1X_1+\dots+h_kX_k)$ is a group isomorphism between $\H$ and $\R^k$. Upon identifying $\H$ and $\R^k$ in this way, for any $w\in\W$ we can consider $f_w:\R^k\to\R^k$ defined by $f_w(h):=f(wh)$. This map is of class $C^1$ and
\[
\nabla f_w(h)=(X_1f(wh),\dots,X_kf(wh)).
\]
We have $|\nabla f_w-M|<\ep$ which, if $\ep$ is small enough,  implies that $f_w$ is a $C^1$ diffeomorphism of $\R^k$: see e.g. the argument in \cite[3.1.1]{federer}\footnote{The careful reader will notice that the argument in \cite[3.1.1]{federer}  works also when the parameter $\de$ introduced therein is $+\infty$.}. This concludes the proof of Step 1; we notice also that, possibly reducing $\ep$, there exists $c>0$ such that (see again in \cite[3.1.1]{federer})
\begin{equation}\label{eq:coerciva}
|f(wh_1)-f(wh_2)|=|f_w(h_1)-f_w(h_2)|\geq c|h_1-h_2|\quad\forall\:h_1,h_2\in\R^k.
\end{equation}

{\em Step 2: $\f$ is intrinsic Lipschitz.} By Remark \ref{rem:altriconi} it is enough to prove that
\[
\mathrm{gr\ }\f\cap D(x,\H,\be)=\{x\}\qquad\text{for any }x\in\G
\]
for a suitable $\be>0$ that we will choose in a moment. 

Let then $x\in \mathrm{gr\ }\f$ be fixed; consider $x'\in D(x,\H,\be)$, so that $x'=xy$ for some $y\in D(\H,\be)$. By definition, there exists $h\in\H$ such that 
\[
d(0,h^{-1}y)=d(h,y)\leq \be d(h,0).
\]
Denoting by $L$ the Lipschitz constant of $f$ we deduce using \eqref{eq:coerciva} that
\begin{align*}
|f(x')|=&|f(xhh^{-1}y)-f(x)|\\
\geq& |f(xh)-f(x)| -|f(xhh^{-1}y)-f(xh)|\geq c\|h\| - Ld(h,y) \geq (\tilde c -\be L)d(0,h)
\end{align*}
for some $\tilde c>0$. In particular, if $\be$ is small enough, one can have $f(x')=0$ only if $h=0$, which immediately gives $x'=x$. This concludes the proof.
\end{proof}


We can eventually  prove Theorem \ref{teo:intersezione}.

\begin{proof}[Proof of Theorem \ref{teo:intersezione}]
By property $\prop_k$ and Remark \ref{rem:propkvssottogruppi}, the vertical plane $\W:=T_x\Si_1\cap\dots\cap T_x\Si_k$ admits a complementary horizontal homogeneous subgroup $\H$. One can then easily conclude using  Theorems \ref{teo:FSJGAmisurafinita} and \ref{teo:interse'grafLip}.
\end{proof}

\bibliographystyle{acm}

\end{document}